 \documentclass[final,3p,times,review]{elsarticle}

\usepackage{amsmath}
\usepackage{amssymb}
\usepackage{amsthm}

\newcommand{\Om}{\Omega}
\newcommand{\Ga}{\Gamma}
\renewcommand{\vec}[1]{\mathbf{#1}}

\newtheorem{theorem}{Theorem}[section]

\newtheorem{lemma}[theorem]{Lemma}

\newtheorem{proposition}[theorem]{Proposition}
\newtheorem{definition}[theorem]{Definition}
\newtheorem{erg}[theorem]{Result}
\newtheorem{beo}[theorem]{Remark}

\DeclareMathOperator*{\essinf}{ess\,inf}

\DeclareMathOperator*{\mass}{mass}

\journal{Journal of Differential Equations}

\begin{document}

\begin{frontmatter}



\title{Nontrivial Periodic Solutions of Marine Ecosystem Models of $N$-$DOP$ type}


\author[cau]
{C.~Roschat\corref{cor1}\fnref{dfg}}
\ead{cro@informatik.uni-kiel.de}
\author[cau]
{T.~Slawig}
\address[cau]
{Department of Computer Science, Kiel Marine Science - Centre for Interdisciplinary Marine Science, Christian-Albrechts-Universit\"at zu Kiel,
24098 Kiel, Germany
}
\cortext[cor1]{Corresponding author.
}
\fntext[dfg]{Supported by the DFG Cluster Future Ocean, Grant No. CP1338.}

\begin{abstract}
We investigate marine ecosystem models of $N$-$DOP$ type with regard to nontrivial periodic solutions. The elements of this important, widely-used model class typically consist of two coupled advection-diffusion-reaction equations. The corresponding reaction terms are divided into a linear part, describing the transformation of one model variable into the other, and a bounded nonlinear part. Additionally, the model equations conserve the mass contained in the system, i.e. the masses of both variables add up to a constant total mass. In particular, the trivial function is a periodic solution. 
In this paper, we prove that there is at least one periodic solution for every prescribed total mass. 
The proof makes use of the typical properties of $N$-$DOP$ type models by combining results from monotone operator theory and a fixed point argument. In the end, we apply the theorem to the $PO_4$-$DOP$ model, an $N$-$DOP$ type model which is well-known and often used.
\end{abstract}


\end{frontmatter}


\section{Introduction}
Marine ecosystems are described via mathematical models. The $N$-$DOP$ type models, consisting of two coupled advection-diffusion-reaction equations, form one of the most important model classes. Characteristically, these models reflect the transformation of one substance into the other, quantified by a transformation rate $\lambda$.
Further biogeochemical processes are represented by nonlinear, bounded reaction terms. As in most models, the ocean dynamics, specified by advection and diffusion, are supposed to be equal in both equations. This is a reasonable assumption because in applications, ocean dynamics are pre-computed in order to avoid simulating both ocean and biochemical models simultaneously. 
Another important feature of $N$-$DOP$ type models is the conservation of mass, i.e. the total mass does not change with respect to time.

%
$N$-$DOP$ type models are widely spread because they are relatively simple and thus well-suited for testing purposes (cf. \cite{pr13,kwo09, and95}). In \cite{kri10, kri12}, an assessment on the basis of real data indicates that models of $N$-$DOP$ type can often compete with more complicated ones. In \cite{pa05, pa06}, two slightly differing $N$-$DOP$ type models are extended by a third equation in order to reveal the interaction between iron concentration and oceanic processes. The numerous applications of $N$-$DOP$ type models show their importance and relevance. 

In most applications, the model equations are solved periodically because of the observational data the solutions are compared to. Data are usually averaged over several years in order to smooth out one-off effects. To verify and validate numerically obtained solutions, theoretical results about their existence are helpful and desired. 
However, in opposite to transient solutions of $N$-$DOP$ type model equations (cf. Roschat et al.~\cite{ro14}) periodic ones have not been the object of investigation so far. %

Periodic solvability of semilinear parabolic partial differential equations is a challenging task. The assumptions in most established existence theorems include coercivity and monotonicity or pseudo-monotonicity (cf. e.g. \cite{gaj67, shi97, sat02}). 
However, the mentioned standard approaches do not apply to model equations of $N$-$DOP$ type directly, particularly since the conservation of mass condition prevents the advection-diffusion-reaction operator from being coercive. Furthermore, the model equations, being designed as closed systems, have no sources or sinks (inhomogeneities) and thus are solved by the trivial function 0. This periodic solution represents an empty ecosystem without any reactions which is of no interest for neither marine biologists nor mathematicians. Thus, it is necessary to exclude the trivial solution.  

We overcame these obstacles by developing a new proof for periodic solvability, individually adapted to the setting of $N$-$DOP$ type models. The proof is based on a classical theorem about existence of periodic solutions, applied in different solution spaces, and the Schauder Fixed Point Theorem.

The paper is structured as follows: In the next section, we introduce both classical and weak formulations of the $N$-$DOP$ type model equations and formulate our main theorem about periodic solvability. The third section contains the proof of the main result preceded by some important preliminaries. In the last section, we exemplarily show that the main theorem holds for the $PO_4$-$DOP$ model, an $N$-$DOP$ type model of Parekh et al. \cite{pa05}.

\section{Problem formulation and main result}
\subsection{General assumptions}
Let 
$T>0$ and $\Omega\subset\mathbb{R}^3$ be an open, bounded set with a Lipschitz boundary\footnote{For a definition see e.g. \cite{tr02}, Section~2.2.} $\Gamma:=\partial\Omega$. $\eta(s)$ denotes the outward-pointing unit normal vector in $s\in\Ga$. 
Suppose that $\vec v\in L^\infty(0,T; H^1(\Om)^3)$ has the properties ${\rm div}(\vec v(t))=0$ in $L^2(\Om)$ and $\vec v(t)\cdot\eta=0$ in $L^2(\Ga)$, each for almost every $t\in [0,T]$. Let $\kappa\in L^\infty(0,T;L^{\infty}(\Om))$ with
$\kappa_{\min}:=\essinf_{(x,t)\in \Om\times(0,T)}\kappa(x,t)>0$. 
 
For each $j\in\{1,2\}$, the model's reaction terms 
\begin{displaymath}
d_j:L^2(0,T;L^2(\Om))^2\to L^2(0,T;L^2(\Om))\quad\text{and}\quad b_j: L^2(0,T;L^2(\Om))^2\to L^2(0,T;L^2(\Ga))
\end{displaymath}
are generated by the indexed families $(d_j(t))_t$ and $(b_j(t))_t$ of operators
\begin{displaymath}
d_j(t): L^2(\Om)^2\to L^2(\Om)\quad\text{and}\quad b_j(t):L^2(\Om)^2\to L^2(\Ga)
\end{displaymath}
via $d_j(y,x,t):=d_j(y)(x,t):=d_j(t)(y(t))(x)$ and $b_j(y,x,t):=b_j(y)(x,t):=b_j(t)(y(t))(x)$.

\subsection{The model equations}\label{sec:modelequations}
An ecosystem model of $N$-$DOP$ type has the form 
 \begin{align}
\label{eq:zustand}
\begin{array}{rcll}
\displaystyle \partial_ty_1+{\rm div}(\vec vy_1)-{\rm div}(\kappa\nabla  y_1)-\lambda y_2+d_1(y_1,y_2)&=&0\quad\text{in }\Om\times[0,T]
\\
\displaystyle \partial_ty_2+{\rm div}(\vec vy_2)-{\rm div}(\kappa\nabla  y_2)+\lambda y_2+d_2(y_1,y_2)&=&0\quad\text{in }\Om\times[0,T]\\
\displaystyle \nabla y_j\cdot(\kappa\eta)+ b_j(y_1,y_2)&=&0\quad\text{in }\Ga\times[0,T],
j=1,2.
\end{array}
\end{align}
As to the interpretation, the first equation determines the concentration of a nutrient $N$, the second describes dissolved organic phosphorus ($DOP$). The term $\lambda y_2$ models the amount of $y_2$ that is transformed (``remineralized'') into $y_1$.
 
Biological and numerical ecosystem models usually assume either homogeneous Neumann boundary conditions or none at all. However, the conservation of mass condition, formulated in Eq.~\eqref{eq:masseerhaltung} below, might require other, possibly nonlinear boundary conditions, indicated by $b_j$. The mass is formalized via the integral with respect to $\Om$.  
\begin{definition}\label{defi:mass}
Let $s>0$.
 The function 
 \begin{displaymath}
 \mass:L^1(\Om)^s\to \mathbb{R},\quad\mass(y):=\sum_{j=1}^s\int_{\Omega}y_jdx 
\end{displaymath}
relates any vector of functions on $\Om$ to its total mass in $\Om$.
\end{definition}

The existence result proved in this paper refers to weak solutions. A standard weak formulation for the boundary value problem \eqref{eq:zustand} is 
\begin{align*}\label{eq:schwach_speziell}
\int_0^T\{\langle y_1^{\prime}(t),w_1(t)\rangle_{H^1(\Om)^*} +B(y_1,w_1;t) +(-\lambda y_2(t)+d_1(y,.\,,t),w_1(t))_{L^2(\Om)}
 + (b_1(y,.\,,t),w_1(t))_{L^2(\Gamma)}\}dt&=0\\
 \int_0^T\{\langle y_2^{\prime}(t),w_2(t)\rangle_{H^1(\Om)^*} +B(y_2,w_2;t) +(\lambda y_2(t)+d_2(y,.\,,t),w_2(t))_{L^2(\Om)}
 + (b_2(y,.\,,t),w_2(t))_{L^2(\Gamma)}\}dt&=0
 \end{align*}
for all test functions $w_1,w_2\in L^2(0,T;H^1(\Om))$. The derivatives $y_j^{\prime}$ are assumed to be elements of $L^2(0,T;H^1(\Om)^*)$. The time-dependent bilinear form $B:H^1(\Om)\times H^1(\Om)\times[0,T]\to\mathbb{R}$ is defined by 
\begin{equation*}
 B(u,v;t):=\int_\Om(\kappa(t)\nabla u\cdot\nabla v)dx+  \int_\Om{\rm div}(\vec v(t) u)vdx \quad\text{for all } u,v\in H^1(\Om).
 \end{equation*}
$B$ is well-defined because of the first statement of Lemma~\ref{lem:operatorB} in the next section. 
In case $B$ is applied to time-dependent functions $\alpha,\beta\in L^2(0,T;H^1(\Om))$ we will write $B(\alpha,\beta;t)$ instead of $B(\alpha(t),\beta(t);t)$. 

As usual, we will interpret the weak formulation as operator equations. 
Abbreviating $X:=L^2(0,T;H^1(\Om))$ and $X^*=L^2(0,T;H^1(\Om)^*)$, we define the operators
\begin{align*}
 B:X
\to 
X^*, \,\, &\langle B(z),v\rangle_{X^*}:=\int_0^TB(z,v;t)dt,
\qquad \qquad\qquad
\pm\lambda Id: X\to X^*,\,\,
  \langle \pm\lambda z,v\rangle_{X^*}:=\int_0^T\int_\Om\pm\lambda z(t)v(t)dxdt
\\
& F_j:X^2
\to 
X^*, \,\,
  \langle F_j(y),v\rangle_{X^*}:=\int_0^T\{
  -(d_j(y,.\,,t),v(t))_{L^2(\Om)}
  -(b_j(y,.\,,t),v(t))_{L^2(\Ga)}\}dt,\,\, j\in\{1,2\},
 \end{align*} 
 for all $z,v\in X$, $y\in X^2$. 
A simple estimation shows that all of these operators are well-defined. Although slightly imprecise, we use the name $Id$ for the second operator 
because it is the standard embedding of $X$ into $X^*$.

Thus, solving $N$-$DOP$ type model equations actually means to find a solution $y=(y_1,y_2)$ of
\begin{align}\label{eq:problem}
 y_1^{\prime}+B(y_1)-\lambda y_2&=F_1(y)\nonumber\\
  y_2^{\prime}+B(y_2)+\lambda y_2&=F_2(y)\\
 y(0)&=y(T)\nonumber.
\end{align}
A suitable solution space for each component turns out to be $W(0,T;H^1(\Om)):=\{w\in X;\,w^{\prime}\in X^*\}$.
The periodicity condition is well-defined because of the embedding $W(0,T;H^1(\Om))\hookrightarrow C([0,T];L^2(\Om))$.
%
%
%
%
\subsection{Main result}
This section is dedicated to the main result of this paper. The proof will follow in the next section.
\begin{theorem}\label{sa:maintheorem}
Let $C,\lambda>0$.
Assume that the reaction terms $d$ and $b$ are continuous and that there are quadratically integrable functions $M_d\in L^2(0,T;L^2(\Om))$, $M_b\in L^2(0,T;L^2(\Ga))$ with
 \begin{equation}\label{eq:totalbound}
\max_j|d_j(y,x,t)|\leq M_d(x,t)\quad\text{ and }\quad\max_j|b_j(y,x,t)|\leq M_b(x,t)
\end{equation}
for almost all $t\in[0,T]$, $x\in \Om$ or $\Ga$, respectively, and all $y\in L^2(0,T;L^2(\Om))^2$. Further, suppose that the conservation of mass condition
\begin{equation}\label{eq:masseerhaltung}
\sum_{j=1}^2\left(\int_{\Omega}d_j(y,x,t)dx+\int_{\Ga}b_j(y,s,t)ds\right)=0
\end{equation}
holds. Hence, the problem \eqref{eq:problem} has a periodic solution $y\in W(0,T;H^1(\Om))^2$ with $\mass(y(t))=C$ for all $t\in[0,T]$. 
\end{theorem}

\section{Proof}

\subsection{Preliminaries}

In the first part of the proof, we use results from monotone operator theory. Therefore, we recapitulate some relevant items. Proofs and further information can be found in Zeidler~
\cite[Chapter~23]{zei90} or Gajewski et al.~\cite{gaj67}.

An \emph{evolution triple} $(V,H,V^*)$ consists of a real and separable Hilbert space $H$ and a real, reflexive and separable Banach space $V$ that is continuously embedded and lies dense in $H$. By means of the theorem of Fr\'echet-Riesz, every element of $H$ can be identified with an element of $H^*$. Furthermore, $H^*$ is embedded in $V^*$ by restriction. Shortly, these relations are indicated by the notation $V\subset H \subset V^*$. 
For any evolution triple, the space 
\begin{displaymath}
W(0,T;V):=\{y\in L^2(0,T;V);\,y^{\prime}\in L^2(0,T;V^*)\}
\end{displaymath}
is continuously embedded in $C([0,T];H)$. Thus, an element of $W(0,T;V)$ can be evaluated in every $t\in[0,T]$. The following theorem collects some important facts about evolution triples.
\begin{theorem}\label{satz:hauptsatz} Let $y,v\in W(0,T;V)$. Then, the following properties are valid:
\begin{enumerate}
  \item The map $t\mapsto\|y(t)\|_{H}^2$ is differentiable almost everywhere with $\frac{d}{dt} \|y(t)\|_{H}^2=2\langle y^{\prime}(t),y(t)\rangle_{V^*}$.
  \item The formula of integration by parts  
  \begin{displaymath}
\int_0^T\langle y^{\prime}(t),v(t)\rangle_{V^*}dt+\int_0^T\langle v^{\prime}(t),y(t)\rangle_{V^*}dt=(y(T),v(T))_H-(y(0),v(0))_H
\end{displaymath}
holds.  In particular, this implies the ``fundamental theorem''
\begin{displaymath}
\int_0^T\langle y^{\prime}(t),y(t)\rangle_{V^*}dt=\frac{1}{2}(\|y(T)\|_{H}^2-\|y(0)\|_{H}^2).
\end{displaymath}
\end{enumerate}
\end{theorem}
We continue with some definitions. Given an evolution triple $(V,H,V^*)$ and $X:=L^2(0,T;V)$, the operator  $A:X\to X^*$ is called \emph{monotone} if 
$\langle Au-Av, u-v\rangle_{X^*}\geq0 \text{ for all } u,v\in X$ and \emph{strictly monotone} if $\langle Au-Av, u-v\rangle_{X^*}>0$ for $u\neq v$. $A$ is said to be \emph{coercive} if $\|u\|_X\to \infty$ implies ${\langle Au,u\rangle}_{X^*}/{\|u\|_X}\to\infty$ and \emph{hemicontinuous} if the map $t\mapsto \langle A(u+tv), w\rangle_{X^*}$ is continuous in $[0,1]$ for all $u,v,w\in X$.

The following theorem based on monotone operator theory is one of the major ingredients of the proof.

\begin{theorem}[Existence theorem of Gajewski et al.~\cite{gaj67}]\label{sa:gajewski}
If $A:X\to X^*$ is a continuous, monotone and coercive operator, the problem
\begin{displaymath}
u^{\prime}+Au=f,\quad u(0)=u(T),
\end{displaymath}
has a solution $u\in W(0,T;V)$ for every $f\in X^*$. If $A$ is strictly monotone, the solution is unique.
\end{theorem}
%
%
Next, we gather some results about the advection-diffusion operator $B$.
\begin{lemma}\label{lem:operatorB}
  The operator $B:L^2(0,T;H^1(\Om))\to L^2(0,T;H^1(\Om)^*)$ is linear and monotone. Furthermore,
   \begin{enumerate}
\item $|B(y,w;t)|\leq C_B\|y(t)\|_{H^1(\Om)}\|w(t)\|_{H^1(\Om)}$
  \item$\kappa_{\min}\|\nabla y(t)\|^2_{L^2(\Om)^3}\leq B(y,y;t)$
\item $B(y,1;t)=0$
\item $B(y+c,w;t)=B(y,w;t)$\,\, for every function $c:[0,T]\to\mathbb{R}$,
\end{enumerate}
each for all $y,w\in L^2(0,T;H^1(\Om))$ and almost every $t$.
\end{lemma}
\begin{proof}
The first two items and the monotonicity of $B$ are proved e.g. by Roschat et al.~\cite{ro14}. The third statement holds because of Gau\ss' divergence theorem and the assumption about the velocity vector $\vec v$:
 \begin{displaymath}
B(y,1;t)=\int_\Om(\kappa(t)\nabla y(t)\cdot\nabla1)dx+  \int_\Om{\rm div}(\vec v(t)y(t))1dx=0+ \int_\Ga(\vec v(t)y(t))\cdot\eta ds=\int_\Ga y(t)(\vec v(t)\cdot\eta) ds=0.
\end{displaymath}

Finally, $B$ is bilinear, $c(t)$ is constant with respect to $x$ and $\vec v(t)$ is divergence free. Thus, we obtain 
\begin{displaymath}
B(c,w;t)=c(t)\left(\int_\Om(\kappa(t)\nabla1\cdot\nabla w(t))dx+ \int_\Om{\rm div}(\vec v(t))w(t)dx\right)=0.
\end{displaymath}
This proves the last statement of the lemma.
\end{proof}

Another important argument in the upcoming proof is the fixed point theorem of Schauder (see e.g. \cite[Thm.~2.A]{zei85}):
\begin{theorem}[Schauder Fixed Point Theorem]
Let $M$ be a nonempty, closed, bounded and convex subset of a Banach space $X$. Suppose $A:M\to M$ is continuous and maps bounded sets to relatively compact sets (i.e. $A$ is a compact operator). Then $A$ has a fixed point.
\end{theorem}

\subsection{Proof of Theorem~\ref{sa:maintheorem}}
The proof of the existence theorem \ref{sa:maintheorem} is divided into two steps. First, the equations are linearized and solved with the help of monotone operator theory. Afterwards, the Schauder Fixed Point Theorem is applied to obtain a solution of the nonlinear problem.

\paragraph{Periodic solution of a linearized problem}
Let $z\in L^2(0,T;L^2(\Om))^2$ be arbitrary. In this first step we show that
\begin{align}\label{eq:problem_lin}
 y_1^{\prime}+B(y_1)-\lambda y_2&=F_1(z)\nonumber\\
  y_2^{\prime}+B(y_2)+\lambda y_2&=F_2(z)\\
 y(0)&=y(T)\nonumber\\
 \mass(y(t))&=C\quad\text{ for all }t\in[0,T].\nonumber
\end{align}
has a unique solution $y=(y_1,y_2)$. To this end, we apply Theorem~\ref{sa:gajewski} twice to different evolution triples. It proves necessary to switch to a solution space in which  the operator $B$ is coercive. 

First, we remark that, by linearization, the two model equations of \eqref{eq:problem_lin} have become decoupled. In particular, it is possible to solve the second equation 
\begin{align*}
  y_2^{\prime}+B(y_2)+\lambda y_2&=F_2(z)\\
 y_2(0)&=y_2(T)\nonumber
\end{align*}
independently of the first. The operator $A:=B+\lambda Id:L^2(0,T;H^1(\Om))\to L^2(0,T;H^1(\Om)^*)$ is linear and therefore hemicontinuous. By means of Lemma~\ref{lem:operatorB}, we obtain the estimate 
\begin{displaymath}
\langle B(y_2)+\lambda y_2, y_2\rangle_{L^2(0,T;H^1(\Om)^*)}\geq\int_0^T\!\!\{\kappa_{\min}\|\nabla y_2(t)\|^2_{L^2(\Om)^3}+\lambda\|y_2(t)\|^2_{L^2(\Om)}\},
dt
\geq\min\{\kappa_{\min},\lambda\}\|y_2\|^2_{L^2(0,T;H^1(\Om))}
\end{displaymath}
which immediately proves that $A$ is coercive and strictly monotone. Hence, Theorem~\ref{sa:gajewski}, applied to the evolution triple $(H^1(\Om), L^2(\Om), H^1(\Om)^*)$,  yields a unique periodic solution $y_2:=y_2(z)\in W(0,T;H^1(\Om))$.

It remains to find a periodic solution $y_1$ of the first equation. The operator $B$ is not coercive in the space $L^2(0,T;H^1(\Om))$ because the lower bound in the second statement of Lemma~\ref{lem:operatorB} only contains the norm of the gradient. 
In the following, we define another solution space in which $B$ is a coercive operator. 

The new evolution triple will be given by $V:=\{y\in H^1(\Om): \mass(y)=0\}$ and $H:=\overline{V}^{L^2(\Om)}$, the closure of $V$ with respect to the $L^2(\Om)$-norm. $V$ is a sub-Hilbert space of $H^1(\Om)$ and therefore reflexive and separable. Furthermore, $\|\,.\,\|_{V}:V\to\mathbb{R};\,y\mapsto\|\nabla y\|_{L^2(\Omega)^3}$ is a norm on $V$, equivalent to the usual $H^1$-norm due to Poincar\'e's inequality (see e.g. Evans \cite[Thm.~5.8.1]{evans}
). By definition, $V$ lies dense in $H$, endowed with the $L^2$-norm, and the embedding is continuous. Therefore, $(V, H, V^*)$ is an evolution triple. In addition, we notice
\begin{beo}\label{beo:masse0inH}
 Let $y\in H$. Then $\mass(y)=0$.
\end{beo}
Indeed, for $y\in H$, there exists a sequence $(y_n)_n\subset V$ with $y_n\to y$ with respect to the $L^2(\Om)$-norm. Since $\mass(y_n)=0$ for all $n$ we conclude
\begin{displaymath}
\mass(y)=\int_{\Omega}ydx=\int_{\Omega}(y-y_n)dx\leq\sqrt{|\Om|}\|y-y_n\|_{L^2(\Om)}\to 0.
\end{displaymath}

In order to find the solution's first component $y_1$, a detour via the sum $S:=y_1+y_2$ becomes necessary. Having obtained $S$ in $W(0,T;H^1(\Om))$, $y_1$ can be defined by the difference of $S$ and $y_2$.
Adding up both model equations suggests that $S$ has to solve
\begin{align}\label{eq:summeS}
 S^{\prime}+B(S)&=F_1(z)+F_2(z)\nonumber\\
   S(0)&=S(T)\\
   \mass(S(t))&=C \text{ for all $t$.}\nonumber
\end{align}
As we will see later, this equation provides the advantage that every solution $S\in W(0,T;V)$ automatically belongs to $W(0,T;H^1(\Om))$ because of the conservation of mass condition \eqref{eq:masseerhaltung}.

In some ecosystem models, the reaction terms fulfill $\sum_{j=1}^2d_j(y,x,t)=0$ and $b_1=b_2=0$ instead of just \eqref{eq:masseerhaltung}. In this case, the right-hand side of the equation for $S$ is zero and, thus, the constant function $S\in W(0,T;H^1(\Om))$ with $S(x,t)=|\Om|^{-1} C$ solves problem \eqref{eq:summeS}. However, for $N$-$DOP$ type models, this is usually not the case.

In order to treat a nontrivial right-hand side, Eq.~\eqref{eq:summeS} is solved on the basis of the evolution triple $(V,H,V^*)$, assuming the homogeneous condition $\mass(S(t))=0$. To this end, we restrict the summands to $L^2(0,T;V)\subset L^2(0,T;H^1(\Om))$ and obtain $B:L^2(0,T;V)\to L^2(0,T;V^*)$ as well as $F_1(z)+F_2(z)\in L^2(0,T;V^*)$.
The restricted $B$ is still hemicontinuous in $L^2(0,T;V)$ and, in addition,
 strictly monotone since
\begin{displaymath}
\langle B(S),S\rangle_{L^2(0,T;V^*)}
\geq\kappa_{\min} \int_0^T\|\nabla S(t)\|_{L^2(\Om)^3}^2dt=\kappa_{\min} \int_0^T\|S(t)\|_V^2dt=\kappa_{\min} \|S\|_{L^2(0,T;V)}^2>0
\end{displaymath}
if $S\in L^2(0,T;V)\setminus\{0\}$. 
This estimate also proves the coercivity of the restricted $B$. Thm.~\ref{sa:gajewski}, applied to the evolution triple $(V,H,V^*)$, yields a unique periodic solution $S\in W(0,T;V)$. Because of Remark~\ref{beo:masse0inH}, $\mass(S(t))=0$ for all $t\in[0,T]$.

Two problems remain to be solved: First, we have to show $S\in W(0,T;H^1(\Om))$, i.e. a larger amount of test functions is allowed. Second, we need $\mass(S(t))=C$ for all $t\in[0,T]$. 

As to the first problem, we remark that the initial value $S(0)\in H$ actually is an element of $L^2(\Om)$. It is well known that there is a transient solution $S_\tau\in W(0,T;H^1(\Om))$ of
\begin{displaymath}
 S_\tau^{\prime}+B(S_\tau)=F_1(z)+F_2(z),\quad S_\tau(0)=S(0)
\end{displaymath}
(see e.g.~Roschat et al.~\cite{ro14}).
Define $S_0\in C([0,T];L^2(\Om))$ by $S_0(t):=S_\tau(t)-|\Omega|^{-1}\mass(S_\tau(t))$ for all $t\in[0,T]$. 
Hence, $S_0$ has the following properties: 

\begin{lemma} $S_0\in L^2(0,T;V)$, $S_0^{\prime}=S^{\prime}_\tau$ and $S_0^{\prime}\in L^2(0,T;H^1(\Om)^*)$.
\end{lemma}
\begin{proof}
The first property holds because, obviously, $S_0\in L^2(0,T;H^1(\Om))$ and 
\begin{displaymath}
\mass(S_0(t))=\mass(S_\tau(t))-\int_{\Om}|\Omega|^{-1}\mass(S_\tau(t))dx=\mass(S_\tau(t))-|\Omega||\Omega|^{-1}\mass(S_\tau(t))=0\quad\text{ for all $t$}.
\end{displaymath}
Since $S_\tau^{\prime}\in L^2(0,T;H^1(\Om)^*)$, the third property follows from the second. By definition of $S_0$, the latter is equivalent to the weak differentiability of the map $\mass(S_\tau):[0,T]\to \mathbb{R}$ 
 with derivative 0.

To show this, let $\varphi\in C^\infty_0(0,T)$. Since the support of $\varphi$ is compact in $(0,T)$, we have $\varphi(0)=\varphi(T)=0$. $\varphi$ can be interpreted as an element of $W(0,T;H^1(\Om))$, constant with respect to $x$. The interpretation of the function $\varphi^{\prime}$ as an element of $L^2(0,T;H^1(\Om)^*)$ especially yields
$\int_0^T\mass(\varphi^{\prime}(t)S_\tau(t))dt=\int_0^T\langle \varphi^{\prime}(t),S_\tau(t)\rangle_{H^1(\Om)^*}dt$. 

Applying integration by parts in $W(0,T;H^1(\Om))$ (cf. Thm.~\ref{satz:hauptsatz}(2)), we obtain
\begin{align*}
-\!\int_0^T\!\!\mass(S_\tau(t))\varphi^{\prime}(t)dt&=-\!\int_0^T\!\!\langle\varphi^{\prime}(t),S_\tau(t) \rangle_{H^1(\Om)^*}dt=\!\int_0^T\!\!\langle S^{\prime}_\tau(t), \varphi(t)\rangle_{H^1(\Om)^*}dt-(\varphi(T),S_\tau(T))_{L^2(\Om)}+(\varphi(0),S_\tau(0))_{L^2(\Om)}\\
&=\!\int_0^T\!\!\langle S^{\prime}_\tau(t), \varphi(t)\rangle_{H^1(\Om)^*}dt
=\!\int_0^T\!\!(B(S_\tau,1;t)+\langle F_1(z)+F_2(z),1\rangle_{H^1(\Om)^*})\varphi(t)dt=0.
\end{align*}
In the last line, we inserted the equation $S_\tau$ solves, applied to the test function $\varphi\in L^2(0,T;H^1(\Om))$, and used that $\varphi(t)$ is independent of $x$. Finally, we employed Lemma~\ref{lem:operatorB}(3) and the conservation of mass condition \eqref{eq:masseerhaltung} which implies
\begin{displaymath}
\langle F_1(z)+F_2(z),1\rangle_{H^1(\Om)^*}=
  \sum_{j=1}^2\left(
  \int_{\Om}{d}_j(z,.\,,t)dx
  +
  \int_\Ga b_j(z,.\,,t)ds\right)=0\quad\text{for almost every $t$.} 
\end{displaymath}
We obtain  
\begin{displaymath}
S_0^{\prime}:=(S_\tau-|\Omega|^{-1}\mass(S_\tau))^{\prime}=S_\tau^{\prime}\in L^2(0,T,H^1(\Om)^*)
\end{displaymath}
which proves the last two claims of the lemma.
\end{proof}

By means of the recent lemma, we can prove that $S_0\in W(0,T;H^1(\Om))$ fulfills the same weak formulation as $S_\tau\in W(0,T;H^1(\Om))$. Since $\mass(S_\tau(t))$ is independent of the spatial coordinate, Lemma~\ref{lem:operatorB}(4) yields indeed
\begin{displaymath}
S_0^{\prime}+B(S_0)=S_\tau^{\prime}+B(S_\tau-|\Om|^{-1}\mass(S_\tau))=S_\tau^{\prime}+B(S_\tau)=
F_1(z)+F_2(z).
\end{displaymath}

In order to verify that $S_0$ is periodic, we prove $S_0=S$. 
Belonging to $L^2(0,T;V)$, the difference $\delta:=S-S_0$ can be inserted in the weak formulations of both $S_0$ and $S$ as a test function. Since these only differ in the space they are formulated in, their difference turns out to be
$\langle \delta^{\prime}(t),\delta(t)\rangle_{H^1(\Om)^*}+B(\delta,\delta;t) = 0$ almost everywhere. 
The statements of Thm.~\ref{satz:hauptsatz}(1) and Lemma~\ref{lem:operatorB}(2) yield
\begin{equation*}
\frac{d}{dt}\|\delta(t)\|_{H}^2\leq- 2\kappa_{\min}\|\nabla\delta(t)\|^2_{L^2(\Omega)^3}\leq0 \quad\text{for almost all $t$.}
\end{equation*}
Consequentially, for every $t\in [0,T]$, Gronwall's lemma leads to
\begin{equation*}
\|\delta(t)\|_{H}^2\leq \exp(0)\|\delta(0)\|^2_{H}=\|S(0)-(S_\tau(0)-|\Om|^{-1}\mass(S_\tau(0)))\|_{H}^2=\|S(0)-S(0)+|\Omega|^{-1}\mass(S(0))\|_{H}^2=0.
\end{equation*}
We used that, by definition, 
 $S_\tau(0)=S(0)\in H$ and $\mass(S(0))=0$ by Remark~\ref{beo:masse0inH}. Therefore, $\delta(t)=0$, i.e. $S(t)=S_0(t)$, for all $t$. 

Thus, $S_0\in W(0,T;H^1(\Om))\cap L^2(0,T;V)$ solves problem \eqref{eq:summeS} except for the condition concerning the mass. In a final step, we add a constant in order to adjust the volume. Define
\begin{displaymath}
S_C:=S_0+|\Om|^{-1}C\in W(0,T;H^1(\Om)).
\end{displaymath}
Obviously, $\mass(S_C(t))=C$ for all $t$. Furthermore, since $|\Om|^{-1}C$ is constant with respect to space and time, $S_C$ is periodic and the equalities $S_C^{\prime}=S_0^{\prime}$ and $B(S_C)=B(S_0)$ hold. Thus, $S_C$ fulfills the same weak formulation as $S_0$. 

Now, we define 
\begin{displaymath}
y_1:=S_C-y_2\in W(0,T;H^1(\Om)).
\end{displaymath}
Then, $y_1$ is periodic and solves the first equation of problem \eqref{eq:problem_lin} because the equations solved by $S_C$ and $y_2$ yield 
\begin{displaymath}
y_1^{\prime}+B(y_1)-\lambda y_2=
S_C^{\prime}-y_2^{\prime}+B(S_C-y_2)-\lambda y_2=
S_C^{\prime}+B(S_C)-(y_2^{\prime}+B(y_2)+\lambda y_2)=F_1(z)+F_2(z)-F_2(z)=F_1(z).
\end{displaymath}
Furthermore, the condition $\mass(y_1(t),y_2(t))=\mass(S_C(t))=C$ holds for all $t\in[0,T]$.

The uniqueness of $(y_1,y_2)$ is an immediate conclusion from the results above. Given two solutions $(y_1,y_2),(\tilde{y}_1,\tilde{y}_2)$ of \eqref{eq:problem_lin}, it holds $y_2=\tilde{y}_2$ as shown above. The difference $\delta:=y_1-\tilde{y}_1$ is a periodic solution of the equation $\delta^{\prime}+B(\delta)=0$ and 
belongs to $L^2(0,T;V)$ because 
\begin{displaymath}
\mass(\delta(t))=\mass(y_1(t))-\mass(\tilde{y}_1(t))=C-\mass(y_2(t))-(C-\mass(\tilde{y}_2(t)))=0\quad\text{for all }t\in [0,T].
\end{displaymath}
Since we have shown above, that equations of this kind (with an arbitrary right-hand side) are uniquely solvable in $L^2(0,T;V)$ and the constant function 0 is a solution, we conclude $\delta=0$. Therefore, the solution of \eqref{eq:problem_lin} is unique. 
\begin{erg}\label{erg:persollin}
 Given a fixed $z\in L^2(0,T;L^2(\Om))^2$,
 the pair $y(z):=(y_1,y_2)\in W(0,T;H^1(\Om))^2$ defines a unique solution of the linearized problem \eqref{eq:problem_lin}. 
\end{erg}
%
%
\paragraph{Periodic solution of the non-linear problem}
In this second step of the proof, we define the map
\begin{displaymath}
A: L^2(0,T;L^2(\Om))^2\to L^2(0,T;L^2(\Om))^2,\quad z\mapsto y(z)
\end{displaymath}
where $y(z)=(y_1,y_2)$ is the unique solution of problem~\eqref{eq:problem_lin}. According to Result~\ref{erg:persollin}, $A$ is well-defined. Obviously, $y$ is a fixed point of $A$ if and only if it is a solution of the original problem \eqref{eq:problem} with $\mass(y(t))=C$ for all $t\in[0,T]$. 

In the following, we will apply the Schauder Fixed Point Theorem  to $A$. 
To start with, we prove a lemma about the estimation of periodic solutions.
%
%
\begin{lemma}\label{lem:abschaetzung}
 Let $W\in\{V, H^1(\Om)\}$, $R\in L^2(0,T;H^1(\Om)^*)
 $ and $\gamma\geq0$. Let $w\in W(0,T;W)$ be a periodic solution of 
$w^{\prime}+B(w)+\gamma w=R.$

If either $\gamma>0$ or $W=V$ there is a constant $K$, only depending on $\gamma,\kappa_{\min}$ and the Poincar\'e constant, such that 
\begin{displaymath}
\|w\|_{L^2(0,T;H^1(\Om))}\leq K\|R\|_{L^2(0,T;H^1(\Om)^*)}.
\end{displaymath}
\end{lemma}
%
%
\begin{proof}
Inserting the element $w\in W(0,T;W)$ itself as a test function, we obtain
\begin{displaymath}
\langle w^{\prime}(t),w(t)\rangle_{W^*}+B(w,w;t) + \gamma\|w(t)\|_{L^2(\Om)}^2 = \langle R(t),w(t)\rangle_{H^1(\Om)^*}\quad\text{for almost every $t$.}
\end{displaymath}
We treat the left-hand side with the same arguments as above and estimate the right by Cauchy-Schwarz and the Cauchy inequality with $\varepsilon$ (see e.g. Evans \cite[B.2]{evans}). Hence
\begin{eqnarray}\label{eq:bewlemmaw}
\frac{1}{2}\frac{d}{dt}\|w(t)\|_{L^2(\Om)}^2+\kappa_{\min}\|\nabla w(t)\|_{L^2(\Om)^3}^2+\gamma\|w(t)\|_{L^2(\Om)}^2
\leq \frac{1}{4\varepsilon}\|R(t)\|_{H^1(\Om)^*}^2+\varepsilon\|w(t)\|_{H^1(\Om)}^2
\end{eqnarray}
for every $\varepsilon>0$. In case $\gamma>0$, we estimate with $\varepsilon_1:=(1/2)\min\{\kappa_{\min}, \gamma\}>0$ 
\begin{eqnarray*} 
\frac{1}{2}\frac{d}{dt}\|w(t)\|_{L^2(\Om)}^2+\frac{1}{2}\min\{\kappa_{\min}, \gamma\}\|w(t)\|_{H^1(\Om)}^2
\leq \frac{1}{2\min\{\kappa_{\min}, \gamma\}}\|R(t)\|_{H^1(\Om)^*}^2.
\end{eqnarray*}
In case $\gamma=0$, we assume $W=V$. Since the norm of the gradient is equivalent to the usual $H^1$-norm on $V$, we have $k\|w(t)\|_{H^1(\Om)}\leq\|\nabla w(t)\|_{L^2(\Om)^3}$ with $k>0$ only depending on the Poincar\'e constant. With $\varepsilon_2:=(1/2)k^2\kappa_{\min}$ we conclude from \eqref{eq:bewlemmaw}: 
\begin{eqnarray*} 
\frac{1}{2}\frac{d}{dt}\|w(t)\|_{L^2(\Om)}^2+\frac{1}{2}k^2\kappa_{\min}\|w(t)\|_{H^1(\Om)}^2
\leq \frac{1}{2k^2\kappa_{\min}}\|R(t)\|_{H^1(\Om)^*}^2.
\end{eqnarray*}
Integrating these equations with respect to $t\in[0,T]$, the first summand vanishes because of the periodicity of $w$ and Thm.~\ref{satz:hauptsatz}(2). Thus, the desired estimate holds with the constant $K:=\max\{k^2\kappa_{\min}, 
\min\{\kappa_{\min}, \gamma\} \}.
$
\end{proof}

In the following, we verify that the operator $A$ fulfills the assumptions of the Schauder Fixed Point Theorem. In a first step, we define a proper domain of definition $M$ for $A$. To this end, we show that the range of $A$ is bounded with respect to the norms of both $L^2(0,T;L^2(\Om))^2$ and $W(0,T;H^1(\Om))^2$, i.e. that $A(z)$ is bounded independently of $z$ for every $z\in L^2(0,T;L^2(\Om))^2$.

As to the second component of $y:=A(z)$, Lemma~\ref{lem:abschaetzung}, applied to $w:=y_2$, $\gamma:=\lambda>0$, $R:=F_2(z)$, yields $\|y_2\|_{L^2(0,T;H^1(\Om))}\leq K_1\|F_2(z)\|_{L^2(0,T;H^1(\Om)^*)}$. The first component was defined by $y_1=S_0+|\Om|^{-1}C-y_2$. Thus, only the boundedness of $S_0$ remains to be shown. The lemma, now applied to $w:=S_0$, $\gamma:=0$, $W:=V$ and $R:=F_1(z)+F_2(z)$, yields $\|S_0\|_{L^2(0,T;H^1(\Om))}\leq K_2\|F_1(z)+F_2(z)\|_{L^2(0,T;H^1(\Om)^*)}$.

Due to the definition of $F_j$ and to the boundedness assumption \eqref{eq:totalbound} there exists a constant $c_1>0$ such that 
\begin{equation*}
\|F_j(z)\|_{L^2(0,T;H^1(\Om)^*)}\leq c_1(\|d_j(z,.\,,.)\|_{L^2(0,T;L^2(\Om))}+\|b_j(z,.\,,.)\|_{L^2(0,T;L^2(\Ga))})
\leq c_1(\|M_d\|_{L^2(0,T;L^2(\Om))}+\|M_b\|_{L^2(0,T;L^2(\Ga))})=:C_1
\end{equation*}
for all $z\in L^2(0,T;L^2(\Om))^2$, $j\in\{1,2\}$.
%
%
This immediately yields the desired estimates $\|y_2\|_{L^2(0,T;H^1(\Om))}\leq K_1C_1$ and $\|y_1\|_{L^2(0,T;H^1(\Om))}\leq \|S_0\|_{L^2(0,T;H^1(\Om))}+\||\Om|^{-1}C\|_{L^2(0,T;H^1(\Om))}+\|y_2\|_{L^2(0,T;H^1(\Om))}\leq 2K_2C_1+(|\Om|^{-1}T)^{\frac{1}{2}}C+K_1C_1=:C_2$

Furthermore, the derivatives $y^{\prime}_1$, $y^{\prime}_2$ can be estimated just like derivatives of transient solutions (see e.g. Evans \cite[Thm.~7.1.3]{evans} or Roschat et al.~\cite{ro14}) since these proofs only use that $y_1$ and $y_2$ are bounded in $L^2(0,T;H^1(\Om))$ and solve a weak formulation. Thus, there is an upper bound  $C_3$ for $y^{\prime}=(y^{\prime}_1$, $y^{\prime}_2)$ in $L^2(0,T;H^1(\Om)^*)^2$, depending on the norms of $y_1,y_2$ 
 and $F_j(z)$ 
 which are all bounded independently of $z$. 
\begin{erg}\label{erg:boundedness}
 Given $z\in L^2(0,T;L^2(\Om))^2$, the value $A(z)$ is bounded in $L^2(0,T;H^1(\Om))^2$, and thus particularly in $L^2(0,T;L^2(\Om))^2$, by $C_4:=\sqrt{C^2_2+(K_1C_1)^2}$. Moreover, $A(z)$ is bounded in $W(0,T;H^1(\Om))^2$ by $C_5:=\sqrt{C_4^2+C^2_3}$. All upper bounds are independent of $z$. 
\end{erg}

In the light of this result, the set
\begin{displaymath}
M:=\{y\in L^2(0,T;L^2(\Om))^2;\,\|y\|_{L^2(0,T;L^2(\Om))^2} \leq C_4\}
\end{displaymath}
turns out to be an appropriate domain of definition for $A$. Indeed, for every $z\in L^2(0,T;L^2(\Om))^2$, especially for every $z\in M$, Result~\ref{erg:boundedness} states $\|A(z)\|_{L^2(0,T;L^2(\Om))^2}\leq \|A(z)\|_{L^2(0,T;H^1(\Om))^2}\leq C_4$. Thus, $A(z)\in M$, i.e. $A:M\to M$ maps $M$ into itself. 

Since $M$ is a closed ball in $L^2(0,T;L^2(\Om))^2$ with a positive radius, it is nonempty, closed, bounded and convex.
To prove the compactness of $A$, let $\tilde{M}\subset L^2(0,T;L^2(\Om))^2$ be a bounded subset of $M$. According to Result~\ref{erg:boundedness}, $A(\tilde{M})$ is a bounded subset of $W(0,T;H^1(\Om))^2$. Since this space is compactly embedded in $L^2(0,T;L^2(\Om))^2$, i.e. the identity map between these spaces is compact (cf. R\r{u}\v{z}i\v{c}ka~\cite{ru10}), $A(\tilde{M})$ 
is a relatively compact subset of $L^2(0,T;L^2(\Om))^2$.

As to the continuity of $A$, we remark that the right-hand sides $F_j:L^2(0,T;L^2(\Om))^2\to L^2(0,T;H^1(\Om)^*)$ are continuous for each $j\in\{1,2\}$ due to the corresponding assumptions about $d_j$ and $b_j$. Given $z,\tilde{z}\in L^2(0,T;L^2(\Om))^2$, the difference $\delta:=A(z)-A(\tilde{z})$ is a periodic solution of
\begin{align*}
 \delta_1^{\prime}+B(\delta_1)-\lambda \delta_2&=F_1(z)-F_1(\tilde{z})\nonumber\\
  \delta_2^{\prime}+B(\delta_2)+\lambda \delta_2&=F_2(z)-F_2(\tilde{z}).
\end{align*}
Concerning the second component, Lemma~\ref{lem:abschaetzung}, applied to $w:=\delta_2$, $\gamma:=\lambda$, $R:=F_2(z)-F_2(\tilde{z})$, yields in particular
\begin{displaymath}
\|A(z)_2-A(\tilde{z})_2\|_{L^2(0,T;L^2(\Om))}\leq K_3\|F_2(z)-F_2(\tilde{z})
\|_{L^2(0,T;H^1(\Om)^*)}.
\end{displaymath} 
As to the first component, the calculation
\begin{displaymath}
\delta_1=A(z)_1-A(\tilde{z})_1=S_0(z)+|\Om|^{-1}C-A(z)_2-(S_0(\tilde{z})+|\Om|^{-1}C-A(\tilde{z})_2)=S_0(z)-S_0(\tilde{z})+(A(\tilde{z})_2-A(z)_2)
\end{displaymath}
 shows that, actually, an estimate for $S_0(z)-S_0(\tilde{z})\in L^2(0,T;V)$ is needed, i.e. for the periodic solution of 
\begin{align*}
  (S_0(z)-S_0(\tilde{z}))^{\prime}+B(S_0(z)-S_0(\tilde{z}))=F_1(z)+F_2(z)-(F_1(\tilde{z})+F_2(\tilde{z})).
  \end{align*}
After re-arranging the terms on the right-hand side, we define $w:=S_0(z)-S_0(\tilde{z})$, $\gamma:=0$, $W:=V$ and $R:=F_1(z)-F_1(\tilde{z})+F_2(z)-F_2(\tilde{z})$. Lemma~\ref{lem:abschaetzung} and the triangle inequality yield
\begin{displaymath}
\|S_0(z)-S_0(\tilde{z})\|_{L^2(0,T;L^2(\Om))}\leq K_4(\|F_1(z)-F_1(\tilde{z})\|_{L^2(0,T;H^1(\Om)^*)} +\|F_2(z)-F_2(\tilde{z})
\|_{L^2(0,T;H^1(\Om)^*)}).
\end{displaymath} 
Combining the previous results we obtain a constant $K_5>0$ with
\begin{displaymath}
\|A(z)-A(\tilde{z})\|_{L^2(0,T;L^2(\Om))^2}^2\leq
K_5(\|F_1(z)-F_1(\tilde{z})\|^2_{L^2(0,T;H^1(\Om)^*)}+\|F_2(z)-F_2(\tilde{z})\|_{L^2(0,T;H^1(\Om)^*)}^2)
\end{displaymath}

The continuity now follows easily. Let $\varepsilon>0$. Due to the continuity of $F_j:L^2(0,T;L^2(\Om))^2\to L^2(0,T;H^1(\Om)^*)$, there exists a $\delta>0$ with
\begin{displaymath}
\|F_j(z)-F_j(\tilde{z})\|_{L^2(0,T;H^1(\Om)^*)}<\frac{\varepsilon}{\sqrt{2K_5}}\quad\text{provided that }\|z-\tilde{z}\|_{L^2(0,T;L^2(\Om))^2}<\delta\quad\text{for both }j \in \{1,2\}.
\end{displaymath}
Together with the estimate for $A(z)-A(\tilde{z})$ this result yields immediately $\|A(z)-A(\tilde{z})\|_{L^2(0,T;L^2(\Om))^2}<\varepsilon$ provided that $\|z-\tilde{z}\|_{L^2(0,T;L^2(\Om))^2}<\delta$.
Thus, $A$ is a compact operator.

Having proved all necessary assumptions, the Schauder Fixed Point Theorem guarantees the existence of a fixed point $y\in M$ of $A$. By definition, $y$ belongs to $W(0,T;H^1(\Om))^2$, is a periodic solution of
\begin{align*}
 y_1^{\prime}+B(y_1)-\lambda y_2&=F_1(y)\nonumber\\
  y_2^{\prime}+B(y_2)+\lambda y_2&=F_2(y)
\nonumber
\end{align*}
and fulfills $\mass(y(t))=C$ for all $t\in[0,T]$. 
\qed
%
\section{Application to the $PO_4$-$DOP$ model by Parekh et al.}
A well-known marine ecosystem model of $N$-$DOP$ type is the $PO_4$-$DOP$ model by Parekh at al.~\cite{pa05}. In this paper, the authors  present a model of the iron concentration in relation to the marine phosphorus cycle. Without the equation for iron, a model of the global phosphorus cycle with the two variables phosphate and dissolved organic phosphorus remains. In the following, we shortly introduce the model equations (cf. also Roschat et al.~\cite{ro14}) and show afterwards that the assumptions of our main theorem are met. 
\subsection{The domain}
The modeled ecosystem is located in a three-dimensional bounded domain $\Om\subseteq\mathbb{R}^3$. $\Om$ is determined by the open, bounded water surface $\Om^{\prime}\subseteq\mathbb{R}^2$ and the depth $h(x^{\prime})>0$ at every surface point $x^{\prime}\in \Om^{\prime}$. The function $h$ is supposed to be continuous and bounded by the total depth of the ocean $h_{max}$. Thus, $\Om:=\{(x^{\prime\!\!},x_3);x^{\prime}\in\Om^{\prime},x_3\in (0,h(x^{\prime}))\}$. 
The boundary $\Ga$ is the union of the surface $\Ga^{\prime}:=\overline{\Om^{\prime}}\times\{0\}$ 
and the boundary inside the water $\{(x^{\prime\!\!},h(x^{\prime}));x^{\prime}\in\Om^{\prime}\}$. The latter is isomorphic to $\Om^{\prime}$ since $h$ is a function.

The domain is separated into two layers, the euphotic, light-flooded zone $\Om_1$ below the surface and the dark, aphotic zone $\Om_2$ beneath. The maximum depth of the euphotic zone is denoted by $\bar{h}_e$. 
The actual depth of the euphotic zone is defined by $h_e(x^{\prime}):=\min\{\bar{h}_e,h(x^{\prime})\}$. We split the surface into the part $\Om_2^{\prime}:=\{x^{\prime}\in \Om^{\prime}; h(x^{\prime})>\bar{h}_e\}$ above the aphotic zone and the rest $\Om^{\prime}_1:=\Om^{\prime}\setminus \Om_2^{\prime}$. The boundary is divided analogously. In summary, the relevant domains of definitions are
\begin{itemize} 
\item the euphotic zone $\Om_1:=
\{(x^{\prime\!\!},x_3);x^{\prime}\in\Om^{\prime},x_3\in (0,h_e(x^{\prime}))\}$,
  \item the aphotic zone $\Om_2:=\{(x^{\prime\!\!},x_3);x^{\prime}\in\Om_2^{\prime},x_3\in (\bar{h}_e,h(x^{\prime}))\}$,
  \item  the euphotic boundary $\Ga_1:=\{(x^{\prime\!\!},h(x^{\prime}));x^{\prime}\in\overline{\Om^{\prime}_1}\}$,
  \item the aphotic boundary $\Ga_2:=\{(x^{\prime\!\!},h(x^{\prime}));x^{\prime}\in\Om_2^{\prime}\}$.
\end{itemize}
%
%
\subsection{The model}
Let the two model variables $y_1:=PO_4$ and $y_2:=DOP$ be assembled in the vector $y$. One important biogeochemical process, typical for all $N$-$DOP$ type models, is the remineralization of $y_2$ into $y_1$ with a remineralization rate $\lambda>0$. Being independent of light, this transformation takes place in the whole domain $\Om$. It is already reflected in the model equations \eqref{eq:zustand}. The remaining processes, represented by the reaction terms $d_j$ and $b_j$, differ according to the layers. In the light-flooded zone, $y_1$ is taken up via photosynthesis. The uptake is modeled in almost every $(x,t)\in\Om\times[0,T]$ by
\begin{align*}
G(y_1,x,t)&:=
\alpha\frac{y_1(x,t)}{|y_1(x,t)|+K_P}\frac{I(x,t)}{|I(x,t)|+K_I}. 
\end{align*}
This expression assumes a maximum uptake $\alpha>0$, limited by the present concentration $y_1(x,t)$ and insolation by means of saturation functions. $K_P,K_I>0$ are the corresponding half saturation constants. Insolation is represented by the non-negative, bounded function $I\in L^{\infty}(0,T;L^{\infty}(\Om))$ which has positive values only in $\Om_1$. 
$G$ can be regarded as a superposition operator of the function
\begin{displaymath}
G:\mathbb{R}\times\Om\times[0,T]\to\mathbb{R}, \quad G(y_1,x,t):=\alpha\frac{y_1}{|y_1|+K_P}\frac{I(x,t)}{|I(x,t)|+K_I}.
\end{displaymath}
Obviously, the functions $x\mapsto G(y_1,x,t)$ and $(x,t)\mapsto G(y_1,x,t)$ are measurable for every fixed $y_1\in \mathbb{R}$ and, if necessary, $t\in[0,T]$ since this is the case for $I$. The function $y_1\mapsto G(y_1,x,t)$ is continuous for almost every $(x,t)\in \Om\times[0,T]$. 
Furthermore, we see easily
\begin{beo}
The estimate $|G(y_1,x,t)|\leq\alpha$ holds for all $y_1\in \mathbb{R}, x\in\Om, t\in[0,T]$. 
\end{beo}
Thus, the results of Appell et al.~\cite[Thms. 3.1,\,3.7]{ap90} (cf. also \cite[Sec. 4.3.3]{tr02}) can be applied twice. First, consider a fixed point of time $t\in[0,T]$. Then, the real function $G(t):\mathbb{R}\times\Om\to \mathbb{R}, \,G(t)(y_1,x):=G(y_1,x,t)$ generates a well-defined and continuous superposition operator $G(t):L^2(\Om)\to L^2(\Om)$. Likewise, the real function $G$ itself generates the continuous superposition operator $G:L^2(0,T;L^2(\Om))\to L^2(0,T;L^2(\Om))$.
In the following, we will write $G(y_1,x,t)$ instead of both $G(y_1(x),x,t)$ and $G(y_1(x,t),x,t)$ if $y_1\in L^2(\Om)$ or $y_1\in L^2(0,T;L^2(\Om))$, respectively.

The model's reaction terms describe that a fraction $\nu\in[0,1]$ of the uptake $G$ is transformed into $y_2$ while the remnants are exported into $\Om_2$. The parameter $\beta>0$ describes the sinking of particles. For almost every $t\in[0,T]$, these processes are represented by the nonlinear coupling terms $d_j(t):L^2(\Om)^2\to L^2(\Om)$ defined by 
\begin{displaymath}
d_1(y,x,t):=
\begin{cases}
 G(y_1,x,t)&\mbox{in $\Om_1$,}\\
-(1-\nu)\int_0^{h_e(x^{\prime})}G(y_1,(x^{\prime\!\!},x_3),t)dx_3\frac{\beta}{\bar{h}_e}\left(\frac{x_3}{\bar{h}_e}\right)^{-\beta-1}&\mbox{in $\Om_2$}
\end{cases}
\qquad\text{and}\qquad
d_2(y,x,t):=
\begin{cases}
 -\nu G(y_1,x,t)&\mbox{in $\Om_1$,}\\
 0&\mbox{in $\Om_2$}
\end{cases}
\end{displaymath}
and the 
 boundary conditions $b_j(t):L^2(\Om)^2\to L^2(\Ga)$, defined by $b_2(t)=0$ and
\begin{displaymath}
b_1(y,x,t):=
\begin{cases}
 -(1-\nu)\int_0^{h_e(x^{\prime})}G(y_1,(x^{\prime\!\!},x_3),t)dx_3&\mbox{for $x=(x^{\prime\!\!},h(x^{\prime}))\in\Ga_1$,}\\
 -(1-\nu)\int_0^{h_e(x^{\prime})}G(y_1,(x^{\prime\!\!},x_3),t)dx_3\left(\frac{h(x^{\prime})}{\bar{h}_e}\right)^{-\beta}&\mbox{for $x=(x^{\prime\!\!},h(x^{\prime}))\in\Gamma_2$,}\\
0&\mbox{for $x=(x^{\prime\!\!},0)\in\Ga^{\prime}$}.
\end{cases}
\end{displaymath}
\begin{proposition}\label{prop:boundedness}
 The reaction terms fulfill the boundedness conditions
\begin{align*}
 \max_j|d_j(y,x,t)|&\leq\max\{\alpha,(1-\nu)\alpha\beta\}&
 \text{for all }y\in L^2(\Om)^2, x\in\Om, t\in[0,T]\\
\max_j|b_j(y,x,t)|&\leq(1-\nu)\alpha\bar{h}_{e}&
\text{for all }y\in L^2(\Om)^2, x\in\Ga, t\in[0,T]
\end{align*}
 \end{proposition}
\begin{proof}
  First we observe that the coordinate indicating depth $x_3$ belongs to $[\bar{h}_e,h(x^{\prime})]$ if $(x^{\prime\!\!},x_3)\in\Om_2\cup\Ga_2$. Given an arbitrary $\gamma>0$, we conclude
\begin{eqnarray}\label{eq:hilf_tiefekleiner1}
\left(\frac{x_3}{\bar{h}_e}\right)^{-\gamma}=\left(\frac{\bar{h}_e}{x_3}\right)^{\gamma}\leq\left(\frac{\bar{h}_e}{\bar{h}_e}\right)^{\gamma}=1\quad\mbox{ for all $(x^{\prime\!\!},x_3)\in\overline{\Om}_2$.}
\end{eqnarray}
Let $y\in L^2(\Om)^2, t\in[0,T]$. We estimate the reaction terms by means of \eqref{eq:hilf_tiefekleiner1} and the remarked boundedness of $G$. First, let $x\in\Om_1$. Then 
$|d_1(y,x,t)|=|G(y_1,x,t)|\leq\alpha$ and $|d_2(y,x,t)|\leq\nu\alpha\leq\alpha$ since $\nu\leq1$.

Given $x\in\Om_2$, we have $h_e(x^{\prime})=\bar{h}_e$ and thus
\begin{displaymath}
|d_1(y,x,t)|\leq(1-\nu)\int_0^{h_e(x^{\prime})}|G(y_1,(x^{\prime\!\!},x_3),t)|dx_3\frac{\beta}{\bar{h}_e}\left(\frac{x_3}{\bar{h}_e}\right)^{-\beta-1}
\leq(1-\nu)\alpha\beta.
\end{displaymath}
Let now $x\in \Ga_2$. Then 
\begin{displaymath}
|b_1(y,x,t)|\leq
(1-\nu)\int_0^{h_e(x^{\prime})}|G(y_1,(x^{\prime\!\!},x_3),t)|dx_3\left(\frac{h(x^{\prime})}{\bar{h}_e}\right)^{-\beta}\leq
(1-\nu)h_e(x^{\prime})\alpha
\leq (1-\nu)\bar{h}_{e}\alpha.
\end{displaymath}
Obviously, the same estimate holds for $|b_1(y,x,t)|$ with $x\in \Ga_1$. Since 
$d_2(t)=0$ in $\Om_2$ and
$b_2(t)=0$ on $\Ga$ the proposition is proved.
\end{proof}
Since the upper bounds are independent of $t$, the proposition ensures that the reaction terms on the spaces of time-dependent functions
\begin{displaymath}
d_j:L^2(0,T;L^2(\Om))^2\to L^2(0,T;L^2(\Om))\quad
\text{ and }\quad
b_j:L^2(0,T;L^2(\Om))^2\to L^2(0,T;L^2(\Ga)),
\end{displaymath}
generated by the families $(d_j(t))_t$, $(b_j(t))_t$ (cf. general assumption), are well-defined for each $j\in\{1,2\}$.

\subsection{Periodic solutions of the $PO_4$-$DOP$ model}
In order to apply the existence theorem to the $PO_4$-$DOP$ model, the corresponding assumptions have to be verified. 
\paragraph{Continuity}
We have already proved in the last section that the uptake function $G:L^2(0,T;L^2(\Om))\to L^2(0,T;L^2(\Om))$ is continuous. In addition, the reaction terms contain the integral of $G$ with respect to the third variable. Therefore, we prove the following general lemma.
\begin{lemma}
 Let $E\in\{\Om,\Om^{\prime}\}$ and $g\in L^{\infty}(0,T;L^{\infty}(E))$. Hence, the operator $F:L^2(0,T;L^2(\Om))\to L^2(0,T;L^2(E))$, given by
 \begin{displaymath}
Fy(x,t):=g(x,t)\int_0^{h_e(x^{\prime})}\!\!y(x^{\prime\!\!},x_3,t)dx_3\quad\text{for all $y\in L^2(0,T;L^2(\Om))$ and almost all $(x,t)\in E\times[0,T]$},
\end{displaymath}
is well-defined and continuous.
\end{lemma}
\begin{proof}
 Let $y\in L^2(0,T;L^2(\Om))$. We investigate the norm of $Fy$ in order to find out that $F$ is well-defined and bounded. First, consider $E=\Om$. With H\"older's inequality and $h_e(x^{\prime})=\min\{\bar{h}_e,h(x^{\prime})\}$ we obtain for the second part of $Fy$
 \begin{equation}\label{eq:hoelder}
\left(\int_0^{h_e(x^{\prime})}\!\!y(x^{\prime\!\!},x_3,t)dx_3\right)^2\leq h_e(x^{\prime})\int_0^{h_e(x^{\prime})}\!\!y(x^{\prime\!\!},x_3,t)^2dx_3
\leq \bar{h}_e\int_0^{h(x^{\prime})}\!\!y(x^{\prime\!\!},x_3,t)^2dx_3
.
\end{equation}
To estimate $Fy$, we express the integral over $\Om$ by the integrals over $\Om^{\prime}$ and $[0,h(x^{\prime})]$. The first part $g$ is bounded by the constant $\|g\|:=\|g\|_{L^{\infty}(0,T;L^{\infty}(E))}$. The second one is estimated by \eqref{eq:hoelder}. Since the upper bound established in \eqref{eq:hoelder} is independent of $\tilde{x}_3\in[0,h(x^{\prime})]$ the corresponding integral vanishes. Finally, the depth function $h$ is bounded by the maximum depth $h_{\max}$. 
These steps lead to the estimate 
\begin{align*}
 \|Fy\|^2_{L^2(0,T;L^2(\Om))}&\leq\|g\|^2\int_0^T\!\!\int_{\Om^{\prime}}\int_0^{h(x^{\prime})}\!\!
 \left(\int_0^{h_e(x^{\prime})}\!\!y(x^{\prime\!\!},x_3,t)dx_3\right)^2
 d\tilde{x}_3dx^{\prime}dt
 \leq
 \|g\|^2\int_0^T\!\!\int_{\Om^{\prime}}h(x^{\prime})
\bar{h}_e\int_0^{h(x^{\prime})}\!\!y(x^{\prime\!\!},x_3,t)^2dx_3
dx^{\prime}dt\\
& \leq
 \|g\|^2h_{\max}\bar{h}_e\int_0^T\!\!\int_{\Om}
y(x^{\prime\!\!},x_3,t)^2
d(x^{\prime\!\!},x_3)dt
= \|g\|^2h_{\max}\bar{h}_e\|y\|_{L^2(0,T;L^2(\Om))}^2.
\end{align*}
In case $E=\Om^{\prime}$, the estimate remains the same, except for the missing integral over $[0,h(x^{\prime})]$. Here, the upper bound for $F$ is thus given by $\|g\|\sqrt{\bar{h}_e}$. 

As a result, $F$ is a well-defined and bounded operator. Being additionally linear, $F$ is continuous.
\end{proof}
The reaction terms $d_j$ and $b_1$ are defined as compositions of $G$ and $F$ with a factor $g$ bounded by 1, cf. \eqref{eq:hilf_tiefekleiner1}. By definition, the boundary integral over $\Ga_1\cup\Ga_2$ corresponds to the integral over $\Om^{\prime}$. Thus, both reaction terms are continuous in the desired spaces.

\paragraph{Boundedness}
The boundedness condition~\eqref{eq:totalbound} is proved in Proposition~\ref{prop:boundedness} since, in particular, constants are quadratically integrable.
%
%
\paragraph{Conservation of mass}
Let $z\in L^2(0,T;L^2(\Om))^2$ be arbitrary. In this paragraph, we prove: 
\begin{equation*}
\sum_{j=1}^2\int_{\Omega}d_j(z,x,t)dx+\int_{\Ga}b_1(z,s,t)ds=0.
\end{equation*}
By definition, the integrals over $\Om$ are equal to
\begin{align*}
\sum_{j=1}^2\int_{\Omega}d_j(z,x,t)dx= \int_{\Om_1}G(z_1,x,t)dx-\int_{\Om_2}(1-\nu)\int_0^{h_e(x^{\prime})}G(z_1,(x^{\prime\!\!},x_3),t)dx_3\frac{\beta}{\bar{h}_e}\left(\frac{x_3}{\bar{h}_e}\right)^{-\beta-1}dx+\int_{\Om_1}-\nu G(z_1,x,t)dx.
\end{align*}
As to the middle term, we obtain by inserting the definition of $\Om_2$ 
\begin{displaymath}
M:=(1-\nu)\int_{\Om_2}\!\int_0^{h_e(x^{\prime})}\!\!G(z_1,(x^{\prime\!\!},x_3),t)dx_3\frac{\beta}{\bar{h}_e}\left(\frac{x_3}{\bar{h}_e}\right)^{-\beta-1}\!\!dx
=(1-\nu)\int_{\Om^{\prime}_2}\!\int_0^{h_e(x^{\prime})}\!\!G(z_1,(x^{\prime\!\!},x_3),t)dx_3\frac{\beta}{\bar{h}_e}\int_{\bar{h}_e}^{h(x^{\prime})}\!\left(\frac{x_3}{\bar{h}_e}\right)^{-\beta-1}\!\!dx_3dx^{\prime\!\!}.
\end{displaymath}
The second integral with respect to $x_3$ can be solved analytically:
\begin{displaymath}
\frac{\beta}{\bar{h}_e}\int_{\bar{h}_e}^{h(x^{\prime})}\left(\frac{x_3}{\bar{h}_e}\right)^{-\beta-1}dx_3
=\left(\frac{1}{\bar{h}_e}\right)^{-\beta}\left[-x_3^{-\beta}\right]_{\bar{h}_e}^{h(x^{\prime})}
=\left(\frac{1}{\bar{h}_e}\right)^{-\beta}\left[-h(x^{\prime})^{-\beta}+\bar{h}_e^{-\beta}\right]
=-\left(\frac{h(x^{\prime})}{\bar{h}_e}\right)^{-\beta}+1.
\end{displaymath}
Using $\Om_1:=\{(x^{\prime\!\!},x_3);x^{\prime}\in\Om^{\prime},x_3\in (0,h_e(x^{\prime}))\}$, $\Om^{\prime}=\Om^{\prime}_1\,\dot\cup\,\Om^{\prime}_2$ and, finally, the definition of the boundary reaction terms, we obtain for $M$:
\begin{align*}
M&=(1-\nu)\int_{\Om^{\prime}_2}\int_0^{h_e(x^{\prime})}G(z_1,(x^{\prime\!\!},x_3),t)dx_3dx^{\prime}
-(1-\nu)\int_{\Om^{\prime}_2}\int_0^{h_e(x^{\prime})}G(z_1,(x^{\prime\!\!},x_3),t)dx_3\left(\frac{h(x^{\prime})}{\bar{h}_e}\right)^{-\beta}dx^{\prime}\\
&=(1-\nu)\int_{\Om_1}G(z_1,x,t)dx
-
(1-\nu)\int_{\Om^{\prime}_1}\int_0^{h_e(x^{\prime})}G(z_1,(x^{\prime\!\!},x_3),t)dx_3dx^{\prime}
+
\int_{\Ga_2}b_1(z,s,t)ds\\
&=(1-\nu)\int_{\Om_1}G(z_1,x,t)dx
+
\int_{\Ga_1}b_1(z,s,t)ds+
\int_{\Ga_2}b_1(z,s,t)ds
.
\end{align*}
 Combining the results, we arrive at
\begin{align*}
\sum_{j=1}^2\int_{\Omega}d_j(z,x,t)dx=(1-\nu) \int_{\Om_1}G(z_1,x,t)dx-M=-\int_{\Ga}b_1(z,s,t)ds.
\end{align*}
This statement is equivalent to the conservation of mass condition.

 \section*{Acknowledgment}
The research of Christina Roschat was supported by the DFG Cluster Future Ocean, Grant No. CP1338.




\bibliographystyle{elsarticle-num}
\bibliography{meinebibliothek}



\end{document}